\documentclass[11pt,a4paper]{amsart}
\title[Tournament limits: Degree distributions, score functions and self--converseness]{Tournament limits: degree distributions, score functions and self--converseness}
\author[Erik ~Th\"ornblad]{Erik Th\"ornblad}
 \address{Department of Mathematics, Uppsala University, Box 480, S-75106 Uppsala, Sweden.}
 \email{erik.thornblad@math.uu.se}
 \date{\today}
\usepackage{amsfonts,amsmath,amssymb,amsthm}
\usepackage{graphics,graphicx}
\usepackage{enumerate}
\usepackage[utf8]{inputenc}
\usepackage{bbm, dsfont}

\usepackage{mathtools}
\mathtoolsset{showonlyrefs}

\numberwithin{equation}{section}

\DeclareSymbolFont{SY}{U}{psy}{m}{n}
\DeclareMathSymbol{\emptyset}{\mathord}{SY}{'306}

\newtheorem{theorem}{Theorem}[section]{\bf}{\it}
\newtheorem{proposition}[theorem]{Proposition}{\bf}{\it}
\newtheorem{lemma}[theorem]{Lemma}{\bf}{\it}
\newtheorem{corollary}[theorem]{Corollary}{\bf}{\it}
\newtheorem{conjecture}[theorem]{Conjecture}{\bf}{\it}

\DeclareMathOperator{\Leb}{Leb}
\newcommand{\E}{\mathbb{E}}
\newcommand{\pP}{\mathbb{P}}
\newcommand{\cT}{\mathcal{T}}
\newcommand{\cP}{\mathcal{P}}
\newcommand{\cS}{\mathcal{S}}
\newcommand{\sS}{\mathsf{S}}
\newcommand{\sC}{\mathsf{C}}
\newcommand{\sT}{\mathsf{T}}
\newcommand{\cD}{\mathcal{T}}
\renewcommand*\d{\mathop{}\!\mathrm{d}}

\theoremstyle{definition}  
\newtheorem{definition}[theorem]{Definition}{\bf}{\it}

\theoremstyle{remark}
\newtheorem{remark}[theorem]{Remark}{\bf}
\newtheorem*{question*}{Question}{\bf}{\it}

\begin{document}

\begin{abstract}
Motivated by known results for finite tournaments, we define and study the score functions of tournament kernels and the degree distributions of tournament limits. Our main theorem completely characterises those distributions that appear as the degree distribution of some tournament limit and those functions that appear as the score function of some tournament kernel. We also show that only the uniform distribution can be realised as the outdegree distribution of a unique tournament limit. Finally we define self-converse tournament limits and kernels and characterise their degree distributions and score functions.
\end{abstract}

\maketitle

% 
% Landau/Moon1968
% Prelims
% Irreducible limits -> really not central anymore -> move this to a small remark somewhere?
% Transitive limits
% Weak convergence
% Measures/distributions -> do not distinguish -> will not talk about distribution functions
% Stars --> test graphs? maybe not
% Degree distributions of tournaments
% Degree distributions of limits Theorem 3.2
% Degree distribution of kernels Def 3.3
% Degree distributions agree Theorem 3.4
% Score function definition Def 5.1 
% Main theorem 3.6
% W_G definition
% Score function discretisation Lemma 3.7 (later on)
% Moment discussion etc
% Uniqueness of score functions Theorem 4.2
% Characterisation of transitive Corollary 4.3
% 
% 
% Converse stuff.... pretty good!

%%%%%%%%%%%%%%%%%%%%%%%%%%%%%%%%%%%%%%%%%%%%%%%%%%%%%
\section{Introduction and main results}
%%%%%%%%%%%%%%%%%%%%%%%%%%%%%%%%%%%%%%%%%%%%%%%%%%%%%

The now established theory of graph limits has a natural translation to tournament limits, see \cite{DiaconisJanson, Thornblad2016b}. In this framework, sequences of tournaments (directed complete graphs) are said to be convergent if all (di--)subgraph homomorphism densities converge. The limit objects arise through an embedding of the set $\cT$ of unlabelled tournaments into a compact metric space $\overline{\cT}$. The boundary set is the set of \emph{tournament limits}. The details of this construction are given in Section \ref{sec:prelim}; see also \cite{DiaconisJanson, Thornblad2016b}. Tournament limits are typically denoted $\Gamma$, and each tournament limit can be represented by a \emph{tournament kernel}, which is a measurable function $W:[0,1]^2 \to [0,1]$ satisfying $W(x,y)+W(y,x)=1$ almost everywhere. 

A \emph{generalised tournament} $G=(V(G),\alpha)$ consists of a set $V(G)=\{1,2,\dots, n\}$ of vertices and a function $\alpha: V(G)\times V(G)\to [0,1]$ such that $\alpha(u,v)+\alpha(v,u)=1$ for all distinct $u,v\in V(G)$ and $\alpha(u,u)=0$ for all $u\in V(G)$. If $\alpha$ only takes values in $\{0,1\}$, then $G$ is called a \emph{tournament}. If $G$ is a tournament, we will call the set of $(u,v)\in V(G)\times V(G)$ such that $\alpha(u,v)=1$ the \emph{edge set} of $G$. The cardinality of $V(G)$ is denoted $v(G)$. Given a vertex $u\in V(G)$, its \emph{outdegree} is defined as $\sum_{v\in V(G)}\alpha(u,v)$ and its \emph{indegree} as $\sum_{v\in V(G)}\alpha(v,u)$. The sequence 
$(d_i)_{i=1}^n$ of outdegrees of the vertices of $G$, listed in some order, is called the \emph{score sequence} of $G$.

Our aim is to translate the notions of score sequences and outdegrees to the setting of tournament limits and kernels, developing a sensible theory in the process. The following classical result can be seen as our motivating result, and we will later see that there is a corresponding result for tournament limits and tournament kernels.

\begin{theorem}[\cite{Landau1953, Moon63}] \label{thm:landaumoon}
A sequence $(d_i)_{i=1}^n$ of non--negative integers (respectively reals) is the score sequence of some tournament (respectively generalised tournament) on $n$ vertices if and only if $\sum_{i\in J} d_i\geq {|J| \choose 2}$ for all subsets $J\subseteq \{1,2,\dots , n\}$, with equality for $J=\{1,2,\dots, n\}$.
\end{theorem}

\begin{remark}
 One may assume that the sequence in Theorem \ref{thm:landaumoon} be non--decreasing; in this case it suffices to check the condition for sets of the type $J=\{1,2,\dots, k\}$. 

Necessity in Theorem \ref{thm:landaumoon} is seen as follows.  Given any subset of size $k$ of the vertices, the induced subtournament must have $\binom{k}{2}$ internal edges, contributing this much to the sum of the score of vertices in that subtournament. Sufficiency is more involved.
\end{remark}

\subsection{Degree distributions}
Given a tournament $G$, denote by $\nu^+(G)$ the (normalised) \emph{outdegree distribution}; i.e. the distribution of the random variable $D/v(G)$ where $D$ is the outdegree of a vertex chosen uniformly at random. Analogously let $\nu^-(G)$ be the normalised indegree distribution. Denote by $\nu(G)$ the joint distribution of $\nu^-(G)$ and $\nu^+(G)$. Since the sum of the indegree and outdegree at a given vertex is always $v(G)-1$, any one of these three distributions determine the other. 
\begin{remark}
 We will not distinguish between a distribution and the corresponding measure, viewing both as elements in the set of $\cP([0,1])$ of probability distributions on $[0,1]$, equipped with the topology of weak convergence. That is, a sequence of probability measures $\mu_n \in \cP([0,1])$ is said to converge (weakly) to a probability measure $\mu\in \cP([0,1])$ if for all continuous bounded $g:[0,1]\to \mathbb{R}$ we have $\int_{[0,1]}g(x)\d\mu_n(x)\to \int_{[0,1]}g(x)\d\mu(x)$.
\end{remark}

The theory of tournament limits is based on homomorphism densities. For this reason, a homomorphism density characterisation of the degree distribution is important. Denote by $\sS_{m,n}$ the digraph with vertex set $A_1\cup A_2\cup A_3$, where $|A_1|=m$, $|A_2|=1$ and $|A_3|=n$ and edges $(u,v)$ whenever $u\in A_i, y\in A_j$ and $i<j$. The stars operate as the ``test graphs'' for the degree distribution and provide a homomorphism density characterisation of the degree distribution. To be precise, the $(m,n)$:th moment of $\nu(G)$ equals the homomorphism density of $\sS_{m,n}$ in $G$, i.e.
\begin{align}
 \int_{[0,1]^2}x^m y^n \d \nu(G)(x,y) = t(\sS_{m,n},G).
\end{align}

\begin{definition}
Let $\Gamma$ be a tournament limit. The \emph{degree distribution}, denoted $\nu(\Gamma)$, of $\Gamma$ is the unique probability distribution on $[0,1]$ with $(m,n)$:th moment equal to $t(\sS_{m,n},\Gamma)$. The marginals $\nu^-(\Gamma)$ and $\nu^+(\Gamma)$ of $\nu(\Gamma)$ are called the \emph{indegree distribution} and \emph{outdegree distribution}, respectively. 
\end{definition}

It is not obvious that the numbers $t(\sS_{m,n},\Gamma)$ must be the moments of some probability distribution. To see this, take any sequence $(G_k)_{k=1}^{\infty}$ of tournaments converging to $\Gamma$ and recall that the numbers $t(\sS_{m,n},G_k)$ in fact are the moments of the degree distribution $\nu(G_k)$. Continuity implies  $t(\sS_{m,n},G_k)\to t(\sS_{m,n},\Gamma)$ as $k\to \infty$, which translates to converge of the moments of the distributions $\nu(G_k)$. The method of moments implies that $\nu(G_k)$ converges to some unique distribution with moments $t(\sS_{m,n},\Gamma)$.

\begin{definition}
 Let $W$ be a tournament kernel. The \emph{degree distribution}, denoted $\nu(W)$, of $W$ is the distribution of the random variable 
\begin{align}
\left(\int_{0}^1W(y,X) \d y, \int_{0}^1W(X,y)\d y \right)
\end{align}
where $X\sim U[0,1]$. The marginals $\nu^-(W)$ and $\nu^+(W)$ of $\nu(W)$ are called the \emph{indegree distribution} and \emph{outdegree distribution}, respectively. 
\end{definition}

\begin{remark}\label{rem:diag}
 The measure $\nu(W)$ is concentrated on $\{(x,y)\in [0,1]^2 \ : \ x+y=1$\}. Hence $\nu$ is determined by any one of its two marginals $\nu^-$ or $\nu^+$. For this reason we will state and prove several results only for $\nu^+$.
\end{remark}

Each tournament kernel represents some tournament limit, so it is desirable that the definitions of the degree distributions agree, in the following sense.

\begin{lemma}\label{thm:distsense}
 If $W$ is a tournament kernel representing some tournament limit $\Gamma$, then $\nu(\Gamma)=\nu(W)$.
\end{lemma}

The proof is another application of the method of moments. It follows by this result that any two equivalent tournament kernels have the same degree distribution.

\subsection{Score functions}
Interpreting the integral $\int_{[0,1]}W(x,y)\d y$ as the \emph{outdegree} at $x$, the following definition is natural.

\begin{definition}
 The \emph{score function} of a tournament kernel $W$ is defined as the function $f:[0,1]\to [0,1]$ given by $f(x)=\int_{0}^1W(x,y)\d y$. 
\end{definition}

Our aim is to characterise those functions which appear as the score functions of some tournament kernel. The following condition corresponds to the condition on the score sequence in Theorem \ref{thm:landaumoon}. 

\begin{definition}
 A function $f:[0,1]\to [0,1]$ is said to satisfy condition \textbf{I} if, for all measurable $B\subseteq [0,1]$, 
\begin{align}
\int_{B}f(x)\d x \geq \frac{\mu(B)^2}{2},
\end{align}
with equality if $\mu(B)=1$.
\end{definition}

It is easy to check that $1-f$ satisfies condition \textbf{I} if and only if $f$ satisfies \textbf{I}. Moreover, if $f$ is the score function of a tournament $W$ and $B\subseteq [0,1]$ is any measurable set, then
\begin{align}
 \int_{B}f(x)\d x=\int_B\int_{[0,1]}W(x,y)\d y \d x\geq \int_{B}\int_B W(x,y)\d y\d x = \frac{\mu(B)^2}{2}
\end{align}
where the final equality comes from the fact that $W(x,y)+W(y,x)=1$ almost everywhere. This shows that a function can be the score function of some tournament kernel only if it satisfies condition \textbf{I}. The converse statement is less trivial and is part of our main theorem.

\subsection{Main results on degree distributions and score functions}

Our main theorem is the following.
\begin{theorem}\label{thm:mainthm}
Let $f:[0,1]\to [0,1]$ be a function. The following statements are equivalent.
\begin{enumerate}[(i)]
 \item There exists a tournament limit $\Gamma$ with outdegree distribution $f(U)$, $U\sim U[0,1]$.
\label{list:1}
 \item There exists a tournament kernel $W$ with score function $f$. \label{list:2}
 \item The function $f$ satisfies \textbf{I}. \label{list:3}
\end{enumerate}
\end{theorem}
Alternatively, \ref{list:1} could be replaced by the equivalent statement that the measure (on $[0,1]$) induced by the outdegree distribution is $\Leb(f^{-1})$.

Note that we already proved \eqref{list:2} $\Rightarrow$ \eqref{list:3}. For the other directions our main ingredients will be discretisations, weak convergence (of score functions and tournament kernels), rearrangements and the Hardy--Littlewood inequality.

A natural question to ask is whether the degree distribution uniquely determines a tournament limit. The examples in Figure \ref{fig:nonunique} demonstrate that this is not always the case. 
\begin{figure}[h]
  \includegraphics{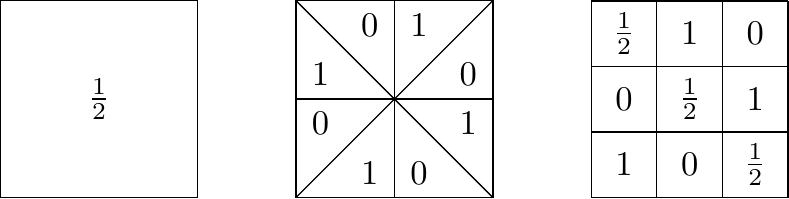}
  \caption{The three squares show three non--equivalent tournament kernels $[0,1]^2\to [0,1]$. (Note that the origin $(0,0)$ is in the top left corner.) Each realises the score function $f(x)=1/2$ and the degree distribution with an atom of mass $1$ at $(1/2,1/2)$.}
 \label{fig:nonunique}
\end{figure}

The cause of non--uniqueness turns out to be the presence of $3$--cycles. Denote by $\sC_3$ the digraph with vertex set $\{1,2,3\}$ and edge set $\{(1,2),(2,3),(3,1)\}$. The unique tournament limit $\Gamma$ satisfying $t(\sC_3,\Gamma)=0$ is called the \emph{transitive limit}. This can be represented by the kernel $W(x,y)=\mathbbm{1}(x\geq y)$, see \cite{Thornblad2016b} for a full characterisation of the transitive limit. If $\Gamma$ is the transitive limit, it is also not difficult to show that $\sT_n\to \Gamma$ as $n\to \infty$, where $\sT_n$ denotes the tournament with vertex set $\{1,2,\dots, n\}$ and edges $(i,j)$ if and only if $i<j$. The tournament $\sT_n$ is known as the transitive tournament on $n$ vertices.

\begin{remark}\label{rem:simple}
The corresponding uniqueness problem for score sequences of tournaments was posed and answered by Avery \cite{Avery1980}. A score sequence that can be realised by a unique (up to isomorphism) tournament is called \emph{simple}. A non--decreasing score sequence $(d_i)_{i=1}^n$ is \emph{irreducible} if 
\begin{align}
 \sum_{i=1}^k d_i > \binom{k}{2}, \quad i=1,\dots, n-1, \qquad \qquad \sum_{i=1}^n d_i = \binom{n}{2}.
\end{align}
Avery \cite{Avery1980} proved that an irreducible score sequence is simple if and only if it is one of $\{0\}$, $\{1,1,1\}$, $\{1,1,2,2\}$ and $\{2,2,2,2,2\}$. 

Furthermore, a tournament is \emph{irreducible} if and only if for all pairs of distinct vertices $u,v$, there exist directed paths between $u$ and $v$ in both directions. One can show that a tournament is irreducible if and only if its non--decreasing score sequence is irreducible. Since any tournament decomposes into a linear order of irreducible subtournaments, see \cite{Moon1968}, Avery's result implies that a tournament is the unique realisation of its score sequence if and only if the score sequences of each of its irreducible subtournaments are one of $\{0\}$, $\{1,1,1\}$, $\{1,1,2,2\}$ and $\{2,2,2,2,2\}$. 

If one takes a sequence of simple score sequences $((d_i^{(n)})_{i=1}^n)_{n=1}^{\infty}$, then the unique tournaments $(G_n)_{n=1}^{\infty}$ with these score sequences will look macroscopically more and more like large transitive tournaments (tournaments with no cycles). In fact, a sequence of tournaments with simple score sequences will always be convergent, and it will always converge to the transitive limit.This can be seen by changing each of its irreducible subtournamnets to a transitive subtournament of the same size. In total this changes $O(n)$ edges, which is not enough to change the limit object.
\end{remark}

The main consequence of the next theorem is that almost all degree distributions can be realised by several different tournament limits, with the one exception (the transitive limit) outlined above. In light of Remark \ref{rem:simple}, this should not be surprising. 

\begin{theorem}\label{thm:unique}
 Let $\Gamma$ be some tournament limit. The following statements are equivalent.
\begin{enumerate}[(i)]
 \item $\Gamma$ is uniquely determined by its degree distribution $\nu(\Gamma)$. \label{un1}
 \item $\Gamma$ is uniquely determined by its outdegree distribution $\nu^+(\Gamma)$.\label{un2}
 \item $\Gamma$ has outdegree distribution $\nu^+(\Gamma)=U[0,1]$.\label{un3}
 \item $\Gamma$ has outdegree distribution $\nu^+(\Gamma)=f(U[0,1])$, where $f:[0,1]\to [0,1]$ is a measure-preserving transformation.\label{un4}
 \item The $(1,1)$:th moment of the degree distribution of $\nu(\Gamma)$ equals $1/6$, i.e. $t(\sS_{1,1},\Gamma)=1/6$.\label{un5}
 \item $\Gamma$ is the transitive limit, i.e. $t(\sC_3,\Gamma)=0$.\label{un6}
\end{enumerate}
\end{theorem}

It holds that
\begin{align}
 t(\sC_3,\Gamma)=t(\sS_{0,0},\Gamma)-3t(\sS_{0,1},\Gamma)+3t(\sS_{1,1})-t(\sC_3,\Gamma)
\end{align}
for any tournament limit $\Gamma$, so the $\sC_3$--density is determined uniquely by the degree distribution. Moreover, since $t(\sS_{0,0},\Gamma)=1$ and $t(\sS_{0,1},\Gamma)=1/2$ (the former is the ``vertex density'', the latter the ``edge density''), we have that $t(\sC_3,\Gamma)=0$ if and only if $t(\sS_{1,1},\Gamma)=1/6$. Since there is a unique limit (see \cite{Thornblad2016b}) with $t(\sC_3,\Gamma)=0$, this shows that there is a unique limit (i.e. the transitive limit) such that the $(1,1)$--moment of its degree distribution is $1/6$. This is the direction  \eqref{un5} $\Leftrightarrow$ \eqref{un6}. The direction \eqref{un1} $\Leftrightarrow$ \eqref{un2}. The directions \ref{un3} $\Leftrightarrow$ \eqref{un4} and  \eqref{un3} $\Rightarrow$ \eqref{un5} are straightforward. The direction \eqref{un6} $\Rightarrow$ \eqref{un2} follows since the transitive limit can be represented by $W(x,y)=\mathbbm{1}(y\leq x)$, the score function of which is $f(x)=x$. The direction \eqref{un6} $\Rightarrow$ \eqref{un3} is straightforward. The only difficult direction is \eqref{un2} $\Rightarrow$ \eqref{un6}, which we deal with in Section \ref{sec:uniqueness}.

Theorem \ref{thm:unique} adds to list of equivalent characterisations of the transitive limit given in \cite{Thornblad2016b}. Furthermore, a consequence of it (and the proof of \eqref{un2} $\Rightarrow$ \eqref{un6}) is the following corollary

\begin{corollary}
 Let $\mu$ be some distribution on $[0,1]$. Then there are either $0$, $1$ or infinitely many tournament limits with degree distribution $\mu$.
\end{corollary}

% 
% \begin{corollary}\label{cor:measpres}
%  Suppose $f:[0,1]\to [0,1]$ satisfies Condition \textbf{I}. Then there is a unique tournament limit $\Gamma$ satisfying $\nu^{+}(\Gamma)=\Leb(f^{-1})$ if and only if $f$ is a measure--preserving transformation. Moreover, this tournament limit is the transitive limit. 
% \end{corollary}

\subsection{Self--converse limits and kernels}
The \emph{converse} of a generalised tournament $G=(V(G),\alpha)$ is the tournament $G'=(V(G),\alpha')$ where $\alpha'(i,j)=1-\alpha(i,j)$ for all $i,j\in V(G)$. Intuitively, $G'$ is obtained by reversing all edges of $G$. Converses of digraphs are defined analogously. Two generalised tournaments $G_1=(V(G_1),\alpha_1)$ and $G_2=(V(G_2),\alpha_2)$ are \emph{isomorphic} if there exists a bijection $\rho : V(G_1) \to V(G_2)$ such that $\alpha_1(i,j)=\alpha_2(\rho(i),\rho(j))$ for all $i,j\in V(G_1)$. A generalised tournament $G$ is said to be \emph{self--converse} if $G$ and $G'$ are isomorphic. The following lemma gives a few equivalent definitions of self--converse tournaments.

\begin{lemma}\label{lem:selfconv}
Let $G=(V(G),\alpha)$ be a tournament. The following are equivalent.
\begin{enumerate}[(i)]
 \item $G$ is self--converse. \label{selfconv1}
 \item $G$ and $G'$ are isomorphic. \label{selfconv2}
 \item $t_{ind}(F,G)=t_{ind}(F',G)$ for any tournament $F$. \label{selfconv3}
 \item $t(F,G)=t(F',G)$ for any digraph $F$. \label{selfconv5}
 \item There exists a bijection $\rho : V(G)\to V(G)$ such that $\alpha(i,j)+\alpha(\rho(i),\rho(j))=n-1$ for all distinct $i,j\in V(G)$. \label{selfconv4}
\end{enumerate}
\end{lemma}

The homomorphism conditions above motivate the following definition.
\begin{definition}
The \emph{converse} of a tournament limit $\Gamma$ is the unique tournament limit $\Gamma'$ with $t(F,\Gamma')=t(F',\Gamma)$ for any digraph $F$. If $\Gamma=\Gamma'$, then $\Gamma$ is said to be \emph{self--converse}. 
\end{definition}
 Equivalently, $\Gamma$ is self--converse if and only if $t(F,\Gamma)=t(F',\Gamma)$ for any digraph $F$.

To see that converses of tournament limits are well--defined, take some tournament limit $\Gamma$ and some sequence $(G_n)_{n=1}^{\infty}$ of tournaments converging to $\Gamma$ as $n\to \infty$. Since $t(F,G')=t(F',G)$ for any digraphs $F,G$, the sequence $(G_n')_{n=1}^{\infty}$ must also be a convergent sequence. Say that it converges to some limit $\Gamma'$. Then
\begin{align}
t(F,\Gamma')= \lim_{n\to \infty}t(F,G_n')=\lim_{n\to \infty}t(F',G_n)=t(F',\Gamma),
\end{align}
which shows that converses are well--defined.

The definition of the converse of a kernel mirrors that for generalised tournaments.

\begin{definition}
 The \emph{converse} of a tournament kernel $W:[0,1]^2\to [0,1]$ is the tournament kernel $W':[0,1]^2\to [0,1]$ defined by $W'(x,y):=1-W(x,y)=W(y,x)$. If $W$ and $W'$ are equivalent (represent the same tournament limit), then $W$ is said to be \emph{self--converse}.
\end{definition}

Converseness is well-behaved under ``representation''. To see this, suppose $W$ represents $\Gamma$. Since $t(F,\Gamma')=t(F',\Gamma)=t(F',W)=t(F,W')$, the converse $W'$ represents $\Gamma'$. This implies that also self--converseness is well--behaved under ``representation''.

\begin{lemma} \label{thm:defsensible}
 Let $\Gamma$ be a tournament limit represented by a kernel $W$. Then $\Gamma'$ is represented by $W'$. Furthermore, $\Gamma$ is self--converse if and only if $W$ is self--converse.
\end{lemma}

It is easy to show that any convergent sequence of self--converse tournaments must converge to a self--converse tournament limit. However, we have not been able to show that any self--converse tournament limit is the limit of a sequence of self--converse tournament, except in some special cases outlined below. The difficulty stems from the fact that a sequence $(G_n)_{n=1}^{\infty}$ converging to some self--converse tournament limit $\Gamma$ need only be \emph{approximately} self--converse, in the sense that $|t(F,G_n)-t(F,G_n')|\to 0$ as $n \to \infty$, for any digraph $F$, so the usual construction via random tournaments $G(n,W)$ will not succeed. Instead we will define a new type of random tournaments which are self--converse while still being comparable to $G(n,W)$. For our definitions to make sense, we need to restrict ourselves to a subclass of self--converse kernels.

\begin{definition}
A tournament kernel is said to be \emph{strongly self--converse} if there exists a measure--preserving transformation $\sigma:[0,1]\to [0,1]$ satisfying $\sigma^2(x)=x$ for almost every $x\in [0,1]$ and $W(x,y)=W(\sigma(y),\sigma(x))$ for almost every $(x,y)\in [0,1]^2$.
\end{definition}

 Any strongly self--converse tournament kernel is self--converse, since $W'(x,y):=W(y,x)$ is equivalent to $W(\sigma(y),\sigma(x))=W(x,y)$. However, we do now know if the converse holds, even in the weaker sense of ``equivalence''.

\begin{question*}
Is every self--converse tournament kernel equivalent to a strongly self--converse tournament kernel?
\end{question*}

The answer is positive if the score function of $W$ is assumed to be almost everywhere injective. Strongly self--converse kernels are easier to handle and we are able to prove the following result.

\begin{proposition}\label{prop:selfconverse}
 If $W$ is a strongly self--converse tournament kernel, then there exists a sequence $(G_n)_{n=1}^{\infty}$ of self--converse tournaments which converges to $W$.
\end{proposition}

However, we do not know of any example  of a (non-strongly) self--converse tournament kernel for which there is no sequence of self--converse tournaments converging it, nor of any example of (non-strongly) self--converse tournament kernels for which there is a sequence of self--converse tournaments converging to it.

The following result is the kernel version of Lemma \ref{lem:selfconv}. 
\begin{lemma}\label{lem:Wselfconv}
Let $W$ be a tournament kernel. The following are equivalent.
\begin{enumerate}[(i)]
 \item $W$ is self--converse. \label{Wselfconv1}
 \item $W$ and $W'$ are equivalent. \label{Wselfconv2}
 \item $t_{ind}(F,W)=t_{ind}(F',W)$ for any tournament $F$. \label{Wselfconv3}
 \item $t(F,W)=t(F',W)$ for any digraph $F$. \label{Wselfconv5}
\end{enumerate}
If Conjecture \ref{con} is true, then the following condition is equivalent to the above.
\begin{enumerate}[(i)]
\setcounter{enumi}{4}
 \item There exist measure--preserving transformations $\sigma_1,\sigma_2 : [0,1]\to [0,1]$ such that $W(\sigma_1(x),\sigma_1(y))+W(\sigma_2(x),\sigma_2(y))=1$ almost everywhere. \label{Wselfconv4}
\end{enumerate}
\end{lemma}

Most of the directions follow either by definition or by Lemma \ref{lem:selfconv} and continuity, so we prove only \eqref{Wselfconv2} $\Leftrightarrow$ \eqref{Wselfconv4}. Note the slight difference between Lemma \ref{lem:selfconv} \eqref{selfconv4} and Lemma \ref{lem:Wselfconv} \eqref{Wselfconv4}. Under the assumption that $W$ be strongly self--converse, then we may assume in  Lemma \ref{lem:Wselfconv} \eqref{Wselfconv4} that $\sigma_1(x)=x$ and $\sigma_2^2(x)=x$ almost everywhere.

The above forms the basis of the theory of self--converse tournaments; our goal is now to characterise the degree distributions of the set of self--converse tournament limits; equivalently the score functions of the set of self--converse tournament kernels. The corresponding results for tournaments and generalised tournaments  were worked out in \cite{Eplett1979, Thornblad2016d}.

\begin{theorem}[\cite{Eplett1979, Thornblad2016d}] \label{thm:Eplett}
 A non--decreasing sequence $(d_i)_{k=1}^n$ of non--negative integers (respectively reals) is the score sequence of some self--converse tournament (respectively generalised tournament) if and only if 
\begin{align}
 \sum_{i\in J}d_i \geq \binom{|J|}{2}
\end{align}
for all $J\subseteq \{1,2,\dots, n\}$ with equality for $J=\{1,2,\dots, n\}$, and moreover $d_i+d_{n+1-i}=n-1$ for $i=1,\dots, n$.
\end{theorem}

With this result in mind, we define the corresponding score function condition and state our main result regarding self--converse limits and kernels.

\begin{definition}
 A measurable function $f:[0,1]\to [0,1]$ is said to satisfy condition \textbf{II} if $f(x)+f(1-x)=1$ for almost all $x\in [0,1]$.
\end{definition}

 \begin{figure}[ht]
\includegraphics{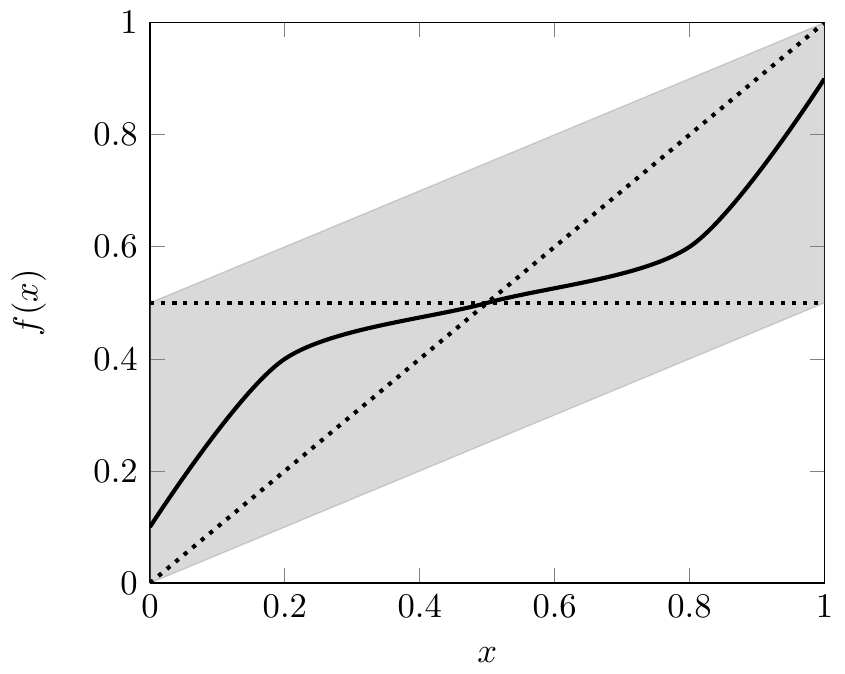}
\caption{A function $f$ (in black) satisfying condition \textbf{I} and \textbf{II}. Any non--decreasing function satisfying \textbf{I} must also satisfy $x/2 \leq f(x) \leq (x+1)/2$ and hence lie inside the gray box. The dotted lines corresponds to the two extreme cases $f(x)=x$ and $f(x)=1/2$. Condition \textbf{I} says that $\int_{[0,r]}(f(x)-x)\d x\geq 0$ with equality for $r=1$. Condition \textbf{II} says that $\{(x,f(x)) \ : \ x\in [0,1]\}$ is fixed under the transformation $(x,y)\mapsto (1-x,1-y)$.}
\end{figure}

\begin{theorem}\label{thm:conv}
 Let $f:[0,1]\to [0,1]$ be a non--decreasing function. The following are equivalent.
\begin{enumerate}[(i)]
 \item There exists a self--converse tournament limit $\Gamma$ with $\nu^{+}(\Gamma)=\Leb(f^{-1})$.\label{thm:convi}
 \item There exists a self--converse tournament kernel $W$ with score function $f$. \label{thm:convii}
 \item $f$ satisfies conditions \textbf{I} and \textbf{II}. \label{thm:conviii}
\end{enumerate}
\end{theorem}

We could also have phrased condition \textbf{II} in terms of the degree distribution. The corresponding condition turns out to be $\nu^+=\nu^-$, and the following alternative characterisation is found.
\begin{proposition}\label{prop:in=out}
Let $\nu=(\nu^-,\nu^+)$ be the degree distribution of some tournament limit. Then there is a self--converse tournament limit $\Gamma$ with degree distribution $\nu(\Gamma)=\nu$ if and only if $\nu^+=\nu^-$.
\end{proposition}

\subsection{Outline}
The rest of this paper is outlined as follows. In Section \ref{sec:prelim} we introduce the necessary background on tournament limits and tournament kernels, including construction of useful random tournaments. In Section \ref{sec:degdist} we prove our initial results regarding the degree distribution and score functions, and provide a short discussion. In Section \ref{sec:uniqueness} we prove Theorem \ref{thm:unique}. Finally, in Section \ref{sec:selfconv} we prove our results on self--converse limits.

%%%%%%%%%%%%%%%%%%%%%%%%%%%%%%%%%%%%%%
\section{Limits, kernels and random tournaments}\label{sec:prelim}
%%%%%%%%%%%%%%%%%%%%%%%%%%%%%%%%%%%%%%

Let $F$ and $G$ be digraphs. A map $\phi:V(F)\to V(G)$ is said to be a \emph{homomorphism} if $(\phi(i),\phi(j))\in E(G)$ for all $(i,j)\in E(F)$, and is said to \emph{preserve non--adjacency} if also $(\phi(i),\phi(j))\notin E(G)$ for all $(i,j)\notin E(F)$. 
For any digraph $G$, denote by  $v(G):=|V(G)|$ its number of vertices. For any two digraphs $F,G$ we define the homomorphism density 
\begin{align}
 t(F,G):=\frac{\hom (F,G)}{v(G)^{v(F)}}
\end{align}
where $\hom(F,G)$ is the number of homomorphisms $V(F)\to V(G)$. The denominator is the total number of maps $V(F)\to V(G)$.
Similarly, the injective and induced homomorphism densities are defined as
\begin{align}
 t_{\text{inj}}(F,G):=\frac{\text{inj} (F,G)}{(v(G))_{v(F)}} \\
 t_{\text{ind}}(F,G):=\frac{\text{ind} (F,G)}{(v(G))_{v(F)}},
\end{align}
where $\text{inj} (F,G)$ denotes the number of injective homomorphisms $V(F)\to V(G)$ and $\text{ind} (F,G)$ denotes the number of injective homomorphisms $V(F)\to V(G)$ that also preserve non--adjacency. It can be shown that the densities $t(\cdot, G)$, $t_{inj}(\cdot, G)$ and $t_{ind}(\cdot, G)$ provide the same information for large digraphs $G$, so we can choose to work with either one of them.

\begin{remark}
 If $F$ is a non--tournament digraph and $G$ a tournament, then $t_{ind}(F,G)=0$. For this reason it is sometimes convenient to work with $t_{ind}$ rather than $t$. 
\end{remark}

The densities (any one of them) above form the basis of the limit theory of tournaments. 
Denote by $\cD$ the set of unlabelled digraphs, and by $\cT\subseteq \cD$ the set of unlabelled tournaments. 
The map $\tau(G):=(t(F,G))_{F\in \cD} \times (v(G)^{-1})\in [0,1]^{\cD}\times [0,1]$ is injective. Denote by $\overline{\cT}$ the closure of $\tau(\cT)$ in $[0,1]^{\cD}\times [0,1]$, observing that this space is compact and metrizable. 
The set $\widehat{\cT}=\overline{\cT}\setminus \tau(\cT)$ is the set of \emph{tournament limits}.
A tournament limit $\Gamma$ can therefore formally be seen as an element of $[0,1]^{\cD}\times \{0\}$. 
A sequence of tournaments $(G_n)_{n=1}^{\infty}$ is said to \emph{converge} to a tournament limit $\Gamma\in \widehat{\cT}$ if $v(G_n)\to \infty$ and $\tau(G_n)\to \Gamma$ as $n\to \infty$. We will always abuse notation and write $G_n \to \Gamma$ instead of $\tau(G_n)\to \Gamma$, and think of $\Gamma$ as the limit of $G_n$ if $G_n$ converges to $\Gamma$.

A \emph{tournament kernel} is a function $W:[0,1]^2\to [0,1]$ such that $W(x,y)+W(y,x)=1$ almost everywhere.
Any tournament kernel $W$ generates a random infinite tournament $G=G(\infty,W)$ as follows. 
Let $X_1,X_2, X_3, \dots$ be an infinite sequence of mutually independent and identically distributed $U[0,1]$ random variables. 
For each $1\leq i < j$, draw an edge from $i$ to $j$ with probability $W(X_i,X_j)$. 
Restricting to the subgraph induced by the first $n$ vertices, one obtains a random tournament $G(n,W)$. 
It can be shown that the limits $\lim_{n\to \infty}t(F,G(n,W))$ exist almost surely for each digraph $F$ and satisfy
\begin{align}
 \lim_{n\to \infty}t(F,G(n,W)) = \int_{[0,1]^{v(F)}} \prod_{(i,j)\in E(F)} W(x_i,x_j) \prod_{i\in V(F)}\d x_i =:t(F,W).
\end{align}

\begin{remark}
 It is possible to consider tournament kernels $\cS^2 \to [0,1]$ for some general probability space $\cS$ instead. This does not give any extra benefit so we will take $\cS=[0,1]$ throughout.
\end{remark}

A tournament limit $\Gamma$ is said to be \emph{represented} by a tournament kernel $W$ if 
\begin{align}
 (t(F,W))_{F\in \cD}\times \{0\} = \Gamma.
\end{align}
  Since $t(F,G(n,W))\to t(F,W)$ almost surely, we have that each tournament kernel represents some tournament limit. Furthermore, it was shown in \cite{Thornblad2016b} that each tournament limit can be represented by some tournament kernel. However, any tournament limit has many different representatives; for instance, applying a measure--preserving transformation $[0,1]\to [0,1]$ to both coordinates of a tournament kernel $W$ does not change the values of the densities $t(F,W)$ for any digraph $F$, so the new kernel still represents the same tournament limit. 

One source of difficulty is the non--uniqueness of tournament kernels which represent the same tournament limit. By analogy to the case of ``undirected'' kernels, we believe the following result to be true, but we are not aware of a proof in the literature.
\begin{conjecture}\label{con}
 If $W$ and $W'$ are equivalent tournament kernels, then there exist measure-preserving transformations $\sigma_1,\sigma_2 : [0,1]\to [0,1]$ such that $W(\sigma_1(x),\sigma_1(y))=W(\sigma_2(x),\sigma_2(y))$ almost everywhere.
\end{conjecture}

%%%%%%%%%%%%%%%%%%%%%%%%%%%%%%%%%%%%%%%%%%%%%%%%%%%%%%%
\section{Results on degree distributions and score functions} \label{sec:degdist}
%%%%%%%%%%%%%%%%%%%%%%%%%%%%%%%%%%%%%%%%%%%%%%%%%%%%%

\begin{proof}[Proof of Lemma \ref{thm:distsense}]
We have 
\begin{align}
 \int_{[0,1]}x^my^nd\nu(\Gamma)
&=t(\sS_{m,n},\Gamma) \\
&=t(\sS_{m,n},W) \\
&=\int_{[0,1]}\left(\int_{[0,1]^m} W(x,y)\d y \right)^m \left(\int_{[0,1]^n}W(z,x)\d z \right)^n \d x \\
&= \E\left[\left(\int_{[0,1]^m} W(U,y)\d y \right)^m \left(\int_{[0,1]^n}W(z,U)\d z \right)^n \right].
\end{align}
The claim now follows by the method of moments.
\end{proof}

For the proof of Theorem \ref{thm:mainthm}, we shall proceed as follows. We have already proved \eqref{list:2} $\Longrightarrow$ \eqref{list:3}. We shall prove \eqref{list:2} $\Longrightarrow$ \eqref{list:1}, \eqref{list:3} $\Longrightarrow$ \eqref{list:2} and \eqref{list:1} $\Longrightarrow$ \eqref{list:3} in that order.

\begin{proof}[Proof of Theorem, \ref{thm:mainthm} \eqref{list:2} $\Longrightarrow$ \eqref{list:1}.]
 Let $\Gamma$ be the graph limit that is represented by $W$. Let $Y\sim \nu^+(\Gamma)$ and let $X\sim U[0,1]$. Then, for any measurable $B\in [0,1]$,
\begin{align}
 \pP\left[ Y\in B \right] = \pP\left[\int_{[0,1]}W(X,y)\d y\in B\right] = \pP\left[f(X)\in B\right],
\end{align}
so $Y\stackrel{d}{=}f(X)$, which proves the claim.

% \begin{align}
% \nu^+(\Gamma)(B)=\nu^{+}(W)(B)= \pP\left[\int_{[0,1]}W(X,y)\d y\in B\right]
% &=\pP[f(X)\in B] \\
% &=\Leb(f^{-1}(B)).
% \end{align}

\end{proof}

To prove Theorem \ref{thm:mainthm}, \eqref{list:3} $\Rightarrow$ \eqref{list:2}, we will make use of a discretisation of $f$ to construct a sequence of generalised tournaments (invoking Theorem \ref{thm:landaumoon}). Each generalised tournament corresponds to a step tournament kernel, and passing to a fine enough subsequence we can ensure that these step functions converge weakly to some tournament kernel with score function $f$. 

Let $G=(\{1,2,\dots, n\},\alpha)$ be a generalised tournament on $n$ vertices. This induces a step tournament kernel $W_G$ defined by
\begin{align}
 W_G(x,y)=\begin{cases}
         \alpha(i,j), \qquad (x,y)\in \left(\frac{i-1}{n},\frac{i}{n}\right]\times\left(\frac{j-1}{n},\frac{j}{n}\right], i\neq j, \\ 
	 1/2, \qquad \quad (x,y)\in \bigcup_{i=1}^n \left((i-1)/n,i/n\right]^2.
        \end{cases}
\end{align}
The boundaries between the subsquares form a null set, so $W_G$ can be defined arbitrarily there. Conversely, if $W$ is a tournament kernel that is constant almost everywhere on each square $[(i-1)/n,i/n]\times [(j-1)/n,j/n]$ for $i,j=1,2\dots, n$, the reverse construction defines a generalised tournament. By this correspondence, the following lemma should not be surprising.

\begin{lemma}\label{lem:score}
 If $f$ satisfies \textbf{I}, then the sequence $(d_i)_{i=1}^{n}$ defined by
\begin{align}
 d_i=n^2 \int_{(i-1)/n}^{i/n}\left( f(x)-\frac{1}{2n}\right)\d x , \qquad i=1,\dots, n,
\end{align}
is a score sequence of some generalised tournament. Moreover, there is a tournament kernel $W_n$ with score function 
\begin{align}
 f_n(x)=n \int_{(i-1)/n}^{i/n}f(y)\d y, \qquad x\in \left( \frac{i-1}{n}, \frac{i}{n} \right], \qquad i=1,\dots, n.
\end{align}
\end{lemma}

\begin{proof}
 For any $J\subseteq \{1,\dots n\}$, let $B_J = \bigcup_{j\in J}[(j-1)/n,j/n]$. Then $|J|=n\mu(B_J)$ and so, since $f$ satisfies \textbf{I},
\begin{align}
 \sum_{i\in J}d_i 
= n^2  \int_{B_J}\left( f(x)-\frac{1}{2n}\right)\d x  
\geq \frac{n^2 \mu(B_J)^2}{2} - \frac{n\mu(B_J)}{2} 
=\frac{|J|^2}{2} - \frac{|J|}{2} 
=\binom{|J|}{2}.
\end{align}
Also note that this implies that $d_i\geq 0$ for any $i=1,\dots, n$. Moreover, for $J=\{1,\dots, n\}$ we have $\sum_{i\in J}d_i=n^2\int_0^1 \left(f(x)-\frac{1}{2n}\right) \d x = \frac{n^2}{2}-\frac{n}{2}=\binom{n}{2} = \binom{|J|}{2}$. 
Therefore the condition in Theorem \ref{thm:landaumoon} is satisfied, so there exists some generalised $G_n$ tournament with score sequence $(d_i)_{i=1}^n$. The tournament kernel $W_n=W_{G_n}$ has score function $f_n$.
\end{proof}

Using this lemma, we can now prove \eqref{list:3} $\Longrightarrow$ \eqref{list:2} in Theorem \ref{thm:mainthm}. We will use basic concepts from weak convergence . A sequence of functions $f_n\in L^p$ is said to converge \emph{weakly} to a function $f\in L^p$ if for all $g\in L^q$ with $1/p+1/q=1$ we have $\int f_n(x)g(x)\d x \to \int f(x)g(x) \d x$.

\begin{proof}[Proof of Theorem \ref{thm:mainthm}, \eqref{list:3} $\Longrightarrow$ \eqref{list:2}.]
 Given $f$, construct the score functions $(f_n)_{n=1}^{\infty}$ and the associated sequence $(W_n)_{n=1}^{\infty}$ of tournament kernels as in Lemma \ref{lem:score}. Note that $f_n\to f$ a.e. as $n\to \infty$. The sequence $(W_n)_{n=1}^{\infty}$ can be seen as a sequence in the the Hilbert space $L_2([0,1]^2)$, more specifically, as a sequence in the closed, convex and bounded subset $\{U\in L_2([0,1]^2) \ : \ U(x,y)+U(y,x)=1, 0\leq U(x,y)\leq 1 \text{ a.e.} \}$. Note that this is precisely the set of tournament kernels. It follows by the Banach-Alouglu theorem that every sequence in this set has a weakly convergent subsequence, and (by weak closure) the weak limit must also lie in the same set. Passing to a subsequence if necessary, we may therefore assume that $(W_n)_{n=1}^{\infty}$ converges weakly to some tournament kernel $W$. 

We show that $f$ converges weakly to the score function of $W$. Take any $g\in L_2([0,1])$. Since $W_n\to W$ weakly, it follows by applying Fubini's theorem, that
\begin{align}
\int_0^1 f_n(x)g(x)\d x 
&= \int_0 ^1 \left(\int_0^1 W_n(x,y)\d y \right)g(x)\d x \\
&= \int_0^1 \int_0^1 g(x)W_n(x,y)\d x\d y \\
&\to \int_0^1 \int_0^1 g(x)W(x,y)\d x\d y\\
&=\int_0^1 \left(\int_0^1 W(x,y)\d y \right)g(x) \d x,
\end{align}
i.e. $f_n\to \int_0^1 W(\cdot, y)\d y$ weakly. But since also $f_n\to f$ a.e., it holds that $f=\int_0^1 W(\cdot,y)\d y$ a.e.. By changing $W$ on a null set, we may assume that the equality holds everywhere, i.e. $f$ is the score function of $W$. 
\end{proof}

To show \eqref{list:1} $\Rightarrow$ \eqref{list:3} we use the theory of decreasing rearrangements and the Hardy--Littlewood inequality, along with the fact that we know that \eqref{list:2} $\Rightarrow$ \eqref{list:3}. 
Let us first recall some basic facts from the theory of decreasing rearrangements. 
For a reference to this, we refer the reader to \cite{BennettSharpley1988}. 

Given a function $f:[0,1]\to [0,1]$, define $\Leb_f:[0,1]\to [0,1]$ by $\Leb_f(t) = \Leb\{x\in [0,1] \ : \ f(x)> t \}$. The \emph{decreasing} rearrangement of the function $f:[0,1]\to [0,1]$ is the function $f^{\ast}:[0,1]\to [0,1]$ defined by $f^{\ast}(t)=\inf\{\lambda \in [0,1] \ : \ \Leb_f(\lambda)\leq t \}$. 
The \emph{increasing rearrangement} of $f$ is the function $f_{\ast}:[0,1]\to [0,1]$ defined by $f_{\ast}(t)=f^{\ast}(1-t)$. 
The Hardy--Littlewood inequality states that
\begin{align}
\int_{0}^1f^{\ast}(x)h_{\ast}(x)\d x \leq \int_{0}^1f(x)h(x)\d x \leq \int_{0}^1f^{\ast}(x)h^{\ast}(x)\d x
\end{align}
for any measurable $f,h:[0,1]\to [0,1]$. We shall only make use of the first inequality. 
In particular, we will take $h(x)=\mathbbm{1}_{\{x\in B\}}$ for any measurable $B\subseteq [0,1]$, whence $h_{\ast}(x)=\mathbbm{1}_{\{x\in [1-\mu(B),1] \}}$, while $f$ will be the score function of some tournament kernel. 
We also mention the result \cite{Ryff1970} that given any measurable $g:[0,1]\to [0,1]$, there exists a measure--preserving $\sigma:[0,1]\to [0,1]$ such that $g^{\ast}(\sigma(x))=g(x)$ for almost all $x\in [0,1]$.

The notion of a decreasing rearrangement is useful to us because the cumulative distribution function of the measure $\nu^+(\Gamma)=\Leb(f^{-1})$ is the function $1-\Leb_f$, and $f^{\ast}(t)=\Leb \{\lambda\in [0,1] \ : \ \Leb_f(\lambda) > t \}$. 
Since a measure on $[0,1]$ is uniquely determined by its cumulative distribution function, we may consider $\Leb_{f^{\ast}}$ instead of $\Leb(f^{-1})$.

\begin{proof}[Proof of Theorem, \ref{thm:mainthm} \eqref{list:1} $\Longrightarrow$ \eqref{list:3}.]
Let $W$ be any kernel representing $\Gamma$, and let $g$ be the score function of $W$. It follows by \eqref{list:2} $\Longrightarrow$ \eqref{list:1} that $1-\Leb_f=1-\Leb_g$, whence $\Leb_f=\Leb_g$ and $f^{\ast}=g^{\ast}$.

Let $B\subseteq [0,1]$ be any measurable set with $\mu(B)=b$ and let $\sigma$ be a measure--preserving map such that $g^{\ast}(\sigma(x))=g(x)$ for almost all $x\in[0,1]$. By the Hardy--Littlewood inequality, $f^{\ast}=g^{\ast}$, the fact that $g^{\ast}(\sigma(x))=g(x)$, the implication \eqref{list:2}  $\Rightarrow$ \eqref{list:3} of Theorem \ref{thm:mainthm} and the fact that $\sigma$ is measure--preserving, it follows that
\begin{align}
 \int_{B}f(x)\d x \geq \int_{0}^1\mathbbm{1}_{\{x\in [1-b,1]\}}f^{\ast}(x)\d x
& =\int_{0}^1\mathbbm{1}_{\{x\in [1-b,1]\}}g^{\ast}(x)\d x \\
&=\int_0^1 \mathbbm{1}_{\{\sigma(x)\in [1-b,1]\}}g^{\ast}(\sigma(x))\d x \\
&=\int_{\sigma^{-1}([1-b,1])}g(x)\d x \\
& \geq \frac{(\mu(\sigma^{-1}([1-b,1]))^2}{2} \\
& = \frac{\mu(B)^2}{2}.
\end{align}
Therefore $f$ satisfies condition \textbf{I}.
\end{proof}

\subsection{Discussion}
 The degree distribution only depends on $f^{\ast}$, so one may restrict oneself to the class of decreasing functions $[0,1]\to [0,1]$. That is, if $\Gamma$ is a tournament limit with degree distribution $\Leb(f^{-1})=\Leb((f^{\ast})^{-1})$, then by Theorem \ref{thm:mainthm} there exists a tournament kernel $W^{\ast}$ with score function $f^{\ast}$. Let $\sigma:[0,1]\to [0,1]$ be a measure--preserving transformation such that $f=f^{\ast}\circ \sigma$. Then $W(x,y):=W^{\ast}(\sigma(x),\sigma(y))$ is a tournament kernel equivalent to $W^{\ast}$ (all the subgraph densities are the same), and moreover it has score function $f$, since
\begin{align}
 \int_{0}^{1}W(x,y)\d y  = \int_0^1 W^{\ast}(\sigma(x),\sigma(y))\d y = \int_0^1 W^{\ast}(\sigma(x),y)\d y = f^{\ast}(\sigma(x))=f(x).
\end{align}
It follows by this construction that if $f_1,f_2$ are two score functions with the same decreasing rearrangement, then there exist equivalent tournament kernels $W_1$ and $W_2$ with score functions $f_1$ and $f_2$ respectively.

Moreover, if $f$ is assumed to be non--decreasing, then condition \textbf{I} simplifies and needs only be checked for sets of the form $B=[0,r]$, $0<r<1$. This follows since, for arbitrary measurable $B\subseteq [0,1]$ with $\mu(B)=r$, we have by monotonicity
\begin{align}
 \int_B f(x)\d x \geq \int_0^r f(x)\d x = \frac{r^2}{2} = \frac{\mu(B)^2}{2}.
\end{align}

Finally we discuss how the score function relates to certain moment problems. If $\Gamma$ is such that $\nu^{+}(\Gamma)=\Leb(f^{-1})$, then it follows by Theorem \ref{thm:distsense} that
\begin{align}
 t(\sS_{m,n},\Gamma)=\int_0^1 f(x)^m (1-f(x))^n\d x.
\end{align}
In particular, the numbers $t(\sS_{m,n},W)$ determine the moments of the function $f$. Using this relationship (or arguing via inclusion--exclusion type formulae that exist for tournaments and tournament kernels), this formula shows that 
\begin{align}
t(\sS_{m,n},\Gamma)=\sum_{\ell=0}^m \binom{m}{\ell}(-1)^{\ell} t(\sS_{0,\ell+n},\Gamma)=\sum_{\ell=0}^n \binom{n}{\ell}(-1)^{\ell}t(\sS_{m+\ell,0})
\end{align}
which is another way to see  that the degree distribution $\nu(\Gamma)$ is determined by its first or second marginals, $\nu^-(\Gamma)$ or $\nu^{+}(\Gamma)$, see also Remark \ref{rem:diag}.

The moments of $f$ determine $f$ up to a measure--preserving transformation. See e.g. \cite[Proposition A.18]{Lovaszbook}, which states that two measurable functions $f,g:[0,1]\to [0,1]$ have the same moments if and only if there exist measure--preserving transformations $\sigma,\sigma':[0,1]\to [0,1]$ such that $f\circ \sigma = g\circ \sigma'$ almost everywhere. Given a function $f:[0,1]\to [0,1]$, we call the sequence $\int_0^1 f(x)^k \d x$, $k=0,1,2,\dots$ its moment sequence. 
The \emph{Hausdorff moment problem} calls for the characterisation those non--negative sequences $(a_k)_{k=0}^{\infty}$ that can be realised as the moments of some function $f:[0,1]\to [0,1]$.

\begin{proposition}[\cite{Lovaszbook}, Proposition A.20]\label{prop:momentchar}
Let $(a_k)_{k=0}^{\infty}$ be a bounded non--negative sequence. The following are equivalent.
\begin{enumerate}[(i)]
 \item $(a_k)_{k=0}^{\infty}$ is the moment sequence of some measurable function $[0,1]\to [0,1]$.
 \item $a_0=1$ and $\sum_{k=0}^m(-1)^k\binom{m}{k}a_{n+k}\geq 0$ for all $n,m\geq 0$.\label{prop:momentchar1}
 \item $a_0=1$ and the infinite matrix $[a_{n+m}]_{n,m\geq 0}$ is positive semi--definite.\label{prop:momentchar2}
\end{enumerate}
\end{proposition}
Notice that $\sum_{k=0}^m(-1)^k\binom{m}{k}a_{n+k} = \int_{0}^1 f(x)^m(1-f(x))^n \d x$ if $(a_k)_{k=0}^{\infty}$ is the moment sequence of $f$.

In view of Proposition \ref{prop:momentchar} and what we have proved so far, we are therefore dealing with a restricted form of the Hausdorff moment problem, where we consider functions $[0,1]\to [0,1]$ that satisfy condition \textbf{I}. 
The following result, which follows from the preceding discussion and Theorem \ref{thm:mainthm}, provides a characterisation of moment sequences of functions $[0,1]\to [0,1]$ that satisfy condition \textbf{I}.
\begin{proposition}\label{prop:Imomentchar}
 Let $(a_k)_{k=0}^{\infty}$ be a bounded non--negative sequence. The following are equivalent.
\begin{enumerate}[(i)]
 \item $(a_k)_{k=0}^{\infty}$ is the moment sequence of some measurable function $f:[0,1]\to [0,1]$ satisfying condition \textbf{I}.
 \item There exists a sequence $(G_n)_{n=0}^{\infty}$ of tournaments such that $v(G_n)\to \infty$ and $t(\sS_{0,k},G_n)\to a_k$ as $n\to \infty$ for any $k\geq 0$.
 \item There exists a tournament limit $\Gamma$ such that $t(\sS_{0,k},\Gamma)=a_k$ for any $k\geq 0$.
 \item There exists a tournament kernel $W$ such that $\int_0^1 \left(\int_{0}^1 W(x,y) \d y\right)^k \d x = a_k$ for any $k\geq 0$.
\end{enumerate}
\end{proposition}
A natural question is whether these moment sequences can be characterised purely algebraically. That is, can one add additional (algebraic) conditions to items \ref{prop:momentchar1} and \ref{prop:momentchar2} in Proposition \ref{prop:momentchar} to make these equivalent to the statements in Proposition \ref{prop:Imomentchar}?
Applying the results of  \cite{Hurlimann2015} would give such an algebraic condition, essentially saying that a sequence $(a_k)_{k=0}^{\infty}$ satisfies any of the conditions in Proposition \ref{prop:Imomentchar} if and only if the zeroes of a certain polynomial determined by $(a_k)_{k=1}^{\infty}$ satisfies a  inequality similar to that in Theorem \ref{thm:landaumoon}. This however is not particularly illuminating and we refrain from pursuing this matter further here.

%%%%%%%%%%%%%%%%%%%%%%%%%%%%%%%%%%%%%%%%%%%%%%%%%%%%%%%
\section{Results on the uniqueness} \label{sec:uniqueness}
%%%%%%%%%%%%%%%%%%%%%%%%%%%%%%%%%%%%%%%%%%%%%%%%%%%%%

We turn now to the question of uniqueness and completing the proof of Theorem \ref{thm:unique}. 

 \begin{remark}
 For finite tournaments, one can change the orientation of any 3--cycle without changing the score sequence, although this would typically change the isomorphism type of the tournament. (In particular, given any two tournaments with the same score sequence, one can transform one into the other by successively reversing the orientation of $3$--cycles.) The idea of the proof of Theorem \ref{thm:unique} is precisely to reverse the orientation of $3$--cycles in $W$, taking care not to change the score function but still changing the kernel to another, non--equivalent kernel.
\end{remark}

\begin{proof}[Proof of Theorem \ref{thm:unique} \eqref{un2} $\Rightarrow$ \eqref{un6}.]
 We prove the contrapositive statement. Suppose $t(\sC_3,\Gamma)>0$. We claim there is a tournament limit $\Gamma_0\neq \Gamma$ such that $\nu(\Gamma)=\nu(\Gamma_0)$. 

Let $W$ be any representative of $\Gamma$. There must exist some $0<\delta < 1$ such that the set
\begin{align}
 A = \{(x,y,z)\in [0,1]^3 \ : \ \delta < W(x,y),W(y,z),W(z,x)<1-\delta \}
\end{align}
satisfies $\Leb(A)>0$. Otherwise $t(\sC_3,W)=0$, a contradiction. Moreover, there exist disjoint intervals $A_1,A_2,A_3\subseteq [0,1]$ such that $\Leb\left((A_1\times A_2 \times A_3) \cap A  \right)>0$. We may assume $\Leb(A_1),\Leb(A_2),\Leb(A_3)<\delta$. Let $B=(A_1\times A_2\times A_3)\cap A$.

We define a family of tournament kernel $W_s:[0,1]^2\to [0,1]$, $s\in [0,1]$ as follows. For any $(x,y,z)\in B$, define
\begin{align}
 W_s(x,y) & :=  W(x,y)+t\cdot \Leb\{ z\in A_3 \ : \ (x,y,z)\in B \} \\
 W_s(y,z) & :=  W(y,z)+t\cdot \Leb\{ x\in A_1 \ : \ (x,y,z)\in B \} \\
 W_s(z,x) & :=  W(z,x)+t\cdot \Leb\{ y\in A_2 \ : \ (x,y,z)\in B \}
\end{align}
and
\begin{align}
 W_s(y,x) & :=  W(y,x)-t\cdot \Leb\{ z\in A_3 \ : \ (x,y,z)\in B \} \\
 W_s(z,y) & :=  W(z,y)-t\cdot \Leb\{ x\in A_1 \ : \ (x,y,z)\in B \} \\
 W_s(x,z) & :=  W(x,z)-t\cdot \Leb\{ y\in A_2 \ : \ (x,y,z)\in B \}
\end{align}
and define $W_s(x,y)=W(x,y)$ for any point $(x,y)\in [0,1]$ where $W_t$ is not yet defined. The disjointedness of $A_1,A_2,A_3$ ensures that the construction is well--defined. This defines a tournament kernel $W_s:[0,1]^2 \to [0,1]$.

Denote by $f$ and $f_s$ the score functions of $W$ and $W_s$, respectively. We claim that $f=f_s$ for all $s\in [0,1]$. Fix some $x\in [0,1]$. If $x\in A_1$, then, by Fubini's theorem,
\begin{align}
 (f-f_t)(x) & =\int_0^1 (W-W_t)(x,y)\d y \\
&= t \int_{A_2} \Leb\{z\in A_3 \ : \ (x,y,z)\in B \} \d y - t \int_{A_3} \Leb\{y\in A_2 \ : \ (x,y,z)\in B \} \d z  \\
& = t\cdot \Leb\{ (y,z)\in A_2\times A_3 \ : \ (x,y,z)\in B\} - t\cdot  \Leb\{ (y,z)\in A_2\times A_3 \ : \ (x,y,z)\in B\} \\
& = 0.
\end{align}
The other cases with $x\in A_2$ and $x\in A_3$ are similar. If $x\in [0,1]\setminus (A_1\cup A_2\cup A_3)$, then clearly $f_s(x)=f(x)$. Hence $f_s=f$, so $W_s$ and $W$ have the same score functions.

Finally we show that we can choose $s\in [0,1]$ such that $W$ and $W_s$ are not equivalent as kernels. We do this by showing that there is some $s\in [0,1]$ for which
\begin{align}
 t(\sC_4,W)\neq t(\sC_4,W_s).
\end{align}

\begin{figure}[ht]
 \centering
\includegraphics{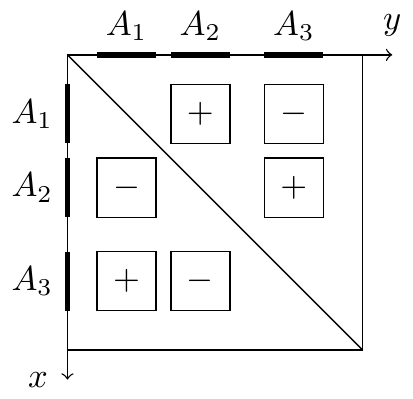}
\caption{An illustration of the construction of $W_{t}$. The 2--dimensional projections of $B$ lie inside the small squares $A_i\times A_j$, $i\neq j$ (but need not fill them up, so the picture is somewhat misleading) and the sign highlights the difference between $W_t$ and $W$.}
\label{fig:epskernel}
\end{figure}

When computing $t(\sC_4,W_t)$, we integrate the function
\begin{align}
 g_s(x_1,x_2,x_3,x_4)=W_s(x_1,x_2)W_s(x_2,x_3)W_s(x_3,x_4)W_s(x_4,x_1)
\end{align}
over $[0,1]^4$. Each such product is a polynomial in $s$ (of degree at most 4)  with constant term equal to
$ W(x_1,x_2)W(x_2,x_3)W(x_3,x_4)W(x_4,x_1)$. This implies that $t(\sC_4,W_s)$ is a polynomial in $s$, of degree at most $4$, with constant term equal to $t(\sC_4,W)$. 

We claim that the coefficient of $s^4$ in $g_s$ must be non--negative. To see this, recall that $W_s$ is equal to $W$ except possibly inside the small squares shown in Figure \ref{fig:epskernel}. If $g_s(x_1,x_2,x_3,x_4)$ is to be of degree $4$, then necessarily $x_1,x_2,x_3,x_4\in A_1 \cup A_2 \cup A_3$. It is straightforward to check all cases and see that the coefficient in front of $s^4$ must be non--negative. Since we have changed $W$ on a non--null set, $g_s(x_1,x_2,x_3,x_4)$ is of degree $4$ on a non--null set of $[0,1]^4$. Therefore
\begin{align}
 t(\sC_4,W_s) = a_4s^4 + a_3s^3+a_2s^2+a_1s + t(\sC_4,W),
\end{align}
where $a_4>0$. We can therefore choose $s_0\in [0,1]$ so that $t(\sC_4,W_{s_0})\neq t(\sC_4,W)$. 

Now let $\Gamma_0$ be the tournament limit represented by $W_{s_0}$. Since $f=f_{s_0}$, we have $\nu(\Gamma)=\nu(\Gamma_0)$. But $\Gamma\neq \Gamma_0$ since $t(\sC_4,\Gamma)\neq t(\sC_4,\Gamma_0)$. This completes the proof.
\end{proof}

% 
% 
% \begin{lemma}\label{lem:trans}
% Let $\Gamma$ be some tournament limit with $\Gamma$ with $\nu^{+}(\Gamma)=\Leb(f^{-1})$ for some score function $f:[0,1]\to [0,1]$. Then $f^{\ast}(x)=1-x$, $x\in [0,1]$ if and only if $\Gamma$ is the transitive limit.
% \end{lemma}
% 
% \begin{proof}
% If $\Gamma$ is the transitive limit, then it can be represented by $W(x,y)=\mathbbm{1}(x\geq y)$, the score function of which is $1-x$.
% 
% Conversely, by Theorem \ref{thm:mainthm} there exists a tournament limit $W$ with score function $f$, i.e. $\int_0^1W(x,y)\d y=1-x$ for almost all $x\in [0,1]$. From this it follows that $t(\sS_{1,1},W)=\int_{[0,1]^3}W(x,y)W(y,z) \d x \d y \d z=1/6$, which implies that $W$ represents the transitive limit.
% \end{proof}
% 
% Corollary \ref{cor:measpres} now follows from Theorem \ref{thm:unique} and Lemma \ref{lem:trans}, along with the result that for any $f:[0,1]\to [0,1]$, there exists a measure--preserving transformation $\sigma:[0,1]\to [0,1]$ such that $f^{\ast}\circ \sigma=f$ (see \cite{Ryff1970}). If $f^{\ast}(x)=1-x$, then $f$ is a measure preserving transformation (since the composition of two measure preserving transformations is another measure preserving transformation), as desired. The converse statement is also straightforward.

% \begin{corollary}
%  An outdegree distribution $\nu^+=\Leb(f^{-1})$ is realised by a unique tournament limit if and only if $f^{\ast}(x)=1-x$ for all $x\in [0,1]$.
% \end{corollary}

It is also proved in \cite{Thornblad2016b} that $W$ represents the transitive tournament limit if and only if $\E\left[\left(\int_0^1W(X,y)\d y\right)^2\right]=1/3$, where $X\sim U[0,1]$. Therefore the transitive tournament limit is determined by the first two moments of its outdegree distribution.

\section{Results on self--converse limits}\label{sec:selfconv}
%%%%%%%%%%%%%%%%%%%%%%%%%%%%%%%%%%%%%%%%%%%%%%%%%%%

\begin{proof}[Proof of Lemma \ref{lem:selfconv}]
\eqref{selfconv1} $\Leftrightarrow$ \eqref{selfconv2}. This is the definition.

\eqref{selfconv5} $\Leftrightarrow$ \eqref{selfconv2}. This follows by the fact that $t(\cdot, G)$ and $t_{ind}(\cdot, G)$ are related via linear combinations, so the homomorphism density of any digraph can be expressed as a linear combination of induced homomorphism densities of digraphs, which is a linear combination of induced homomorphism densities of tournaments.

\eqref{selfconv2} $ \Rightarrow $ \eqref{selfconv3}. Immediate since $t_{ind}(F,G')=t_{ind}(F',G)$ holds for any tournament $G$; now $G$ and $G'$ are isomorphic, so $t_{ind}(F,G')=t_{ind}(F,G)$.

\eqref{selfconv3} $ \Rightarrow $ \eqref{selfconv2}. Take $F=G'$. Then $t_{ind}(G',G)=t_{ind}(G,G)>0$, so $G$ contains a copy of $G'$. Since $v(G)=v(G')$, this implies that $G$ and $G'$ are isomorphic, so $G$ is self--converse.

\eqref{selfconv2} $ \Rightarrow $ \eqref{selfconv4}. The bijection observing the isomorphism between $G$ and $G'$ will satisfy this.

\eqref{selfconv4} $ \Rightarrow $ \eqref{selfconv2}. The bijection $\rho$ will be the required isomorphism between $G$ and $G'$.
\end{proof}

\begin{proof}[Proof of Proposition \ref{prop:selfconverse}]
Let $\sigma:[0,1]\to [0,1]$ be a measure--preserving transformation such that $\sigma^2(x)=x$ for almost every $x\in [0,1]$ and $W(x,y)=W(\sigma(y),\sigma(x))$ for almost every $(x,y)\in [0,1]^2$. Let $X_1,\dots, X_n \sim U[0,1]$ be mutually independent and identically distributed. We construct a sequence of self--converse (random) tournaments $H(2n,W)$ with vertex set $\{v_1,\dots, v_n,w_1,\dots, w_n\}$ as follows.

\begin{itemize}
 \item For $1\leq i<j \leq n$, include the edge $v_iv_j$ with
probability $W(X_i,X_j)$; otherwise include the edge $v_jv_i$. The induced random subtournament on $\{v_1,\dots, v_n\}$ is distributed like $G(n,W)$. 
\item For $1\leq i<j \leq n$, include the edge $w_iw_j$ if and only if the edge $v_jv_i$ was included in the previous step. Note that the probability that the edge $w_iw_j$ is included is precisely
$W(X_j,X_i) = W(\sigma(X_i),\sigma(X_j))$. The induced random subtournament on $\{w_1,\dots, w_n\}$ is therefore distributed like $G(n,W)$. 
\item For any $1\leq i \leq j \leq n$ include the edge $v_iw_j$ with probability $W(X_i,\sigma(X_j)$. For any $1\leq j < i \leq n$, include the edge $v_jw_i$ if and only if the edge $v_iw_j$ was included. Then $v_jw_i$ is included with probability  $W(X_i,\sigma(X_j))=W(X_j,\sigma(X_i))$.
\end{itemize}

The edges of $H(2n,W)$ are only ``weakly dependent'', in the sense that all pairs of distinct edges $(v_i,v_j)$, $(v_k,v_{\ell})$  and all pairs $(w_i,w_j)$, $(w_k,w_{\ell})$  are independent, while edges $(v_i,w_j)$ and $(v_k,w_{\ell})$ are dependent if and only if $i=\ell$ and $j=k$. When sampling $k$ vertices, the (orientation of the) edges induced by these $k$ vertices are mutually independent with probability tending to $1$ as $n\to \infty$. (It is enough to select vertices with different indices.) Moreover, if we sample $k$ vertices with ``independent edges'', then the induced subgraph will be distributed like $G(k,W)$. Denote by $H(2n,W)[k]$ the random induced subtournament of $H(2n,W)$ obtained by sampling $k$ vertices (without replacement). Then, for any tournament $F$ on $k$ vertices,
\begin{align}
 t_{ind}(F,H(2n,W))=\mathbb{P}[F=H(2n,W)[k]] \to \mathbb{P}[F=G(k,W)]=t_{ind}(F,W)
\end{align}
as $n\to \infty$. By \cite{DiaconisJanson}, this implies that $H(2n,W)\to W$ with probability $1$. Since every realisation of $H(2n,W)$ is self--converse, we can extract a sequence of self--converse tournaments converging to $W$.
\end{proof}

\begin{proof}[Proof of Lemma \ref{lem:Wselfconv}]
\eqref{Wselfconv4} $\Rightarrow$ \eqref{Wselfconv2}. 
Let $W_0(x,y) = W(\sigma_1(x),\sigma_1(y))$. Then 
\begin{align}
W_0'(x,y)=W_0(y,x)=W(\sigma_1(y),\sigma_1(x))=1-W(\sigma_2(y),\sigma_2(x))=W(\sigma_2(x),\sigma_2(y)).
\end{align}
 Hence both $W_0$ and $W_0'$ are pull--backs (via $\sigma_1$ and $\sigma_2$ respectively) of $W$, so $W_0$ and $W_0'$ are equivalent (two pullbacks from the same kernel must have the same homomorphism densities, hence be equivalent). Therefore $W_0$ is self--converse. But $W_0$ is equivalent to $W$, so $W$ must also be self--converse. (Another way to see this is that $W_0'(x,y)=W'(\sigma_1(x),\sigma_1(y))$ by the above, so $W'$ is equivalent to $W_0'$, which is equivalent to $W_0$, which is equivalent to $W$.)

\eqref{Wselfconv2} $\Rightarrow$ \eqref{Wselfconv4}. 
Suppose $W$ and $W'$ are equivalent and that there exists  measure--preserving transformations $\sigma_1,\sigma_2 : [0,1]\to [0,1]$ such that $W(\sigma_1(x),\sigma_1(y))=W'(\sigma_2(x),\sigma_2(y))$ almost everywhere. Then  
\begin{align}
 W(\sigma_1(x),\sigma_1(y)) + W(\sigma_2(x),\sigma_2(y)) &= W(\sigma_1(x),\sigma_1(y)) + W'(\sigma_2(y),\sigma_2(x)) \\ &= W(\sigma_1(x),\sigma_1(y))+W(\sigma_1(y),\sigma_1(x))\\ &=1
\end{align}
almost everywhere.

\end{proof}

\begin{proof}[Proof of Theorem \ref{thm:conv}]
 \eqref{thm:convii} $\Leftrightarrow$ \eqref{thm:convi}: This follows by Theorem \ref{thm:defsensible} and the proof of Theorem \ref{thm:mainthm}.
 
 \eqref{thm:convi} $\Longrightarrow$ \eqref{thm:conviii}: Since $\Gamma$ is self--converse and $\sS_{n,0}'=\sS_{0,n}$, we have
\begin{align}
\int_0^1 f(x)^n \d x =  t(\sS_{n,0},\Gamma) = t(\sS_{0,n},\Gamma)=\int_0^1 (1-f(x))^n \d x.
\end{align}
so $f$ and $1-f$ have the same moments, and hence the same decreasing rearrangement. But $f$ is non--decreasing, so its decreasing rearrangement is $f(1-x)$. Since $1-f$ is non--increasing, it is equal to its decreasing rearrangment, so $f(1-x)=1-f(x)$ for almost all $x\in [0,1]$, as desired.

\eqref{thm:conviii} $\Longrightarrow$ \eqref{thm:convi}:  Fix $f$. Take the sequence $(f_n)_{n=1}^{\infty}$ as in Lemma \ref{lem:score}. The associated score sequence $(d_i)_{i=1}^n$ (as in Lemma \ref{lem:score}) satisfies $d_i+d_{n+1-i}=n-1$ for $i=1,2,\dots n$. By Theorem \ref{thm:Eplett}, there is a generalised self--converse tournament $G_n$ with score sequence $(d_i)_{i=1}^n$. By compactness, there is a subsequence of $G_n$ converging to some tournament limit $\Gamma$. For any finite digraph $F$,
\begin{align}
 t(F,\Gamma)=\lim_{n\to \infty}t(F,G_n) = \lim_{n\to \infty}t(F,G_n') = \lim_{n\to \infty}t(F',G_n) = t(F',\Gamma)
\end{align}
which implies that $\Gamma$ is self--converse. Similar to Theorem \ref{thm:mainthm}, we have $\nu^{+}(\Gamma)=\Leb(f^{-1})$.
\end{proof}

\begin{proof}[Proof of Proposition \ref{prop:in=out}]
Suppose $\Gamma$ is self--converse. Then, since $\sS_{n,0}'=\sS_{0,n}$, we have $t(\sS_{n,0},\Gamma)=t(\sS_{0,n},\Gamma)$. Therefore the moments of its indegree and outdegree distributions are equal, so $\nu^+(\Gamma)=\nu^{-}(\Gamma)$.

Conversely, suppose $\nu^+ = \nu^-$. We know that $\nu^+ = \Leb(f^{-1})$ for some $f:[0,1]\to [0,1]$. Similarly one can show that the indegree distribution must be of the form $\nu^- = \Leb((1-f)^{-1})$ for the same $f:[0,1]\to [0,1]$. We may assume that $f$ be non--decreasing (changing to its non--decreasing rearrangment, if necessary). Since $\nu^+ = \nu^-$, we have that $f$ and $1-f$ have the same moments. Hence they have the same non--decreasing rearrangement. This implies $f(x)=1-f(1-x))$ for all $x\in [0,1]$. It follows by Theorem \ref{thm:conv} that there is a self--converse tournament limit $\Gamma$ with $\nu(\Gamma)=\nu$.
\end{proof}

\section*{Acknowledgements}
The author is indebted to Svante Janson, in particular for finding the construction of $W_t$ in the proof of Theorem \ref{thm:unique}.

\bibliographystyle{abbrv}

\end{document}